\documentclass[a4paper,12pt,leqno]{amsart}
\usepackage[a4paper, top=3cm, bottom=3cm, left=3cm, right=3cm]{geometry}
\usepackage[T1]{fontenc}

\usepackage{amsrefs}
\usepackage{amsmath,amsthm,amssymb, enumerate}
\usepackage{verbatim}
\usepackage[noabbrev]{cleveref}

\newtheorem{thm}{Theorem}[section]
\newtheorem{prop}[thm]{Proposition}
\newtheorem{lemma}[thm]{Lemma}
\newtheorem{cor}[thm]{Corollary}
\newtheorem*{thm*}{Main Theorem}

\theoremstyle{definition}
\newtheorem{defn}[thm]{Definition}
\newtheorem{example}[thm]{Example}
\newtheorem{remark}[thm]{Remark}

\newcommand{\Hp}{\mathbf{H}}
\newcommand{\Sp}{\mathbf{S}}

\title{$7$-located locally $5$-large complexes are aspherical}

\begin{document}

\author[K.~Goldman]{Katherine Goldman$^{\dag}$}
	\address{
		Department of Mathematics and Statistics,
		McGill University,
		Burnside Hall,
		805 Sherbrooke Street West,
		Montreal, QC,
		H3A 0B9, Canada}
        
\email{kat.goldman@mcgill.ca}

\thanks{$\dag$ Partially supported by NSF DMS-2402105}

	\author[P.~Przytycki]{Piotr Przytycki$^{\ddag}$}
	
	\address{
		Department of Mathematics and Statistics,
		McGill University,
		Burnside Hall,
		805 Sherbrooke Street West,
		Montreal, QC,
		H3A 0B9, Canada}
	
	\email{piotr.przytycki@mcgill.ca}
	
	\thanks{$\ddag$ Partially supported by NSERC}

	\begin{abstract}
		\noindent
		We prove that $7$-located locally $5$-large simplicial complexes are aspherical.
	\end{abstract}

\maketitle

\section{Introduction}
A simplicial complex is \emph{flag} if every set of vertices pairwise connected by edges spans a simplex. For $k\geq 5$, a simplicial complex is \emph{$k$-large} if it is flag and has no induced cycles of length $4\leq n<k$. A simplicial complex is \emph{locally $k$-large}, if each of its vertex links is $k$-large. 
This notion was introduced by Januszkiewicz and \'Swiątkowski~\cite{JS}, and independently by Haglund \cite{H}, as a simplicial analogue of a locally $\mathrm{CAT}(0)$ (i.e.\ nonpositively-curved) cube complex. They showed that such complexes are abundant in any dimension, and come with interesting automorphism groups. A cornerstone feature is that for $k\geq 6$ they are aspherical. The $1$-skeleta of simply connected locally $6$-large simplicial complexes were studied earlier in graph theory under the name of bridged graphs, see \cite{BC} for a survey.

The boundary of the icosahedron is locally $5$-large, so in order to obtain asphericity under this weaker condition, Osajda introduced an extra hypothesis of $m$-location~\cite{O2} (we will give the definition in a moment). $7$-located locally $5$-large simplicial complexes include many $3$-manifolds, as well as all locally weakly systolic complexes~\cite{HL}, which were studied earlier in \cite{O1,CO}. The properties of $m$-located complexes were investigated in \cite{O2,HL}. A related condition was introduced in \cite{L}.

Maybe the most prominent example of a $7$-located locally $5$-large simplicial complex is the triangulation of the hyperbolic space $\Hp^4$ where each of the vertex links is isomorphic to the boundary of the $600$-cell. The symmetry group of that triangulation is the Coxeter group with Coxeter--Dynkin diagram the linear graph of length $4$ with consecutive labels $5333$. We are interested in this triangulation since the associated Artin group is one of the smallest Artin groups for which the $K(\pi,1)$ conjecture, asking for the contractibility of the associated Artin complex, is still open.

In this paper, we prove the following related result.

\begin{thm*}
Every $7$-located locally $5$-large simplicial complex is aspherical.
\end{thm*}

\noindent \textbf{Acknowledgements.} We are grateful to the anonymous referee for pointing out several omissions in the original version of the article, and for suggesting how to improve the presentation.
We also thank Jingyin Huang for pointing us to \cite{Bj}.

\section{Location}
Let $X$ be a flag simplicial complex. Note that we do not assume $X$ to be locally finite or finite-dimensional.

\begin{defn}
A $k$-\emph{wheel} $W$ in $X$ is an induced subcomplex isomorphic to the cone over the $k$-cycle. We write $W = (v_0, v_1v_2\cdots v_k)$, where the \emph{centre} $v_0$ is the cone vertex and $v_1,\ldots,v_k$ are the consecutive vertices of the \emph{boundary} cycle.

A pair $W=(W_1,W_2)$ of wheels, with $W_1=(v_0, v_1\cdots v_k),W_2=(w_0, w_1\cdots w_\ell)$, is a \emph{$(k,\ell)$-dwheel} if
    \begin{itemize}
        \item $v_k = w_0$, 
        \item $w_\ell = v_0$,
        \item $v_{k-1} = w_{\ell - 1}$, and
        \item either $v_1$ equals $w_1$ or is a neighbour of $w_1$.
    \end{itemize}
    The \emph{boundary} $\partial W$ of the dwheel $W$ is the cycle    
    $v_1\cdots v_{k-1}w_{\ell - 2} \cdots w_1$. (If $v_1 = w_1$, then we discard the redundant $w_1$.) 
    If $v_1 = w_1$, then we say that $W$ is \emph{planar}. 
\end{defn}

\begin{defn}
    Let $m\geq 4$. $X$ is \emph{$m$-located}, if for every dwheel $W=(W_1,W_2)$ with $|\partial W|\leq m$, all the vertices of $W_1\cup W_2$ have a common neighbour in $X$. 
\end{defn}

Here vertices $v\neq w$ are \emph{neighbours} (or \emph{neighbouring}) in $X$  if they span an edge of~$X$.

\begin{example}
Let $X$ be the simplicial complex that is the triangulation of the hyperbolic space $\Hp^4$ into hyperbolic simplices of all dihedral angles $\frac{2\pi}{5}$. Each of the vertex links of $X$ is isomorphic to the boundary $\Sp^3_{600}$ of the $600$-cell, which is $5$-large.
Note that the vertex links of $\Sp^3_{600}$ are isomorphic to the boundary $\Sp^2_{20}$ of the icosahedron. Since each induced $5$-cycle in $\Sp^3_{600}$ and $\Sp^2_{20}$ is the boundary of a $5$-wheel, each $5$-wheel in $X$ can be extended to the join of the $5$-cycle and a triangle~$\Delta$. Furthermore, each induced $6$-cycle in $\Sp^3_{600}$ is the boundary of a $6$-wheel, and so each $6$-wheel in $X$ can be extended to the join of the $6$-cycle and an edge $e$. Hence $X$ does not contain a planar $(5,6)$-dwheel $(W_1,W_2)$ with $W_1\cup W_2$ without a common neighbour, since otherwise appropriate $\Delta$ and~$e$ are disjoint and so $\Delta,e,$ and $v_1=w_1$ span a simplex of dimension $5$ in $X$, which is a contradiction. The $(5,5)$-dwheels are excluded similarly, which implies that $X$ is $7$-located.
\end{example}

\subsection{Disc diagrams}
\label{subsec:discdiagram}
 
A \emph{disc diagram in $X$} is a simplicial map $f\colon D\to X$ from a simplicial disc $D$ that is \emph{nondegenerate}, i.e.\ does not send any edge to a vertex. The \emph{boundary (cycle)} of $f$ is the closed path $f_{|\partial D}$, which we identify with its image $f(\partial D)$ whenever it is embedded. A disc diagram $f\colon D\to X$ is \emph{minimal} if it has minimal area (i.e.\ the number of triangles in~$D$) among all the disc diagrams in $X$ with the same boundary cycle. We say that $f$ is \emph{reduced} if it is locally injective at $D\setminus D^{0}$. The following is a well-known variation of a result by Van Kampen. 

\begin{lemma}[{\cite[Lem~2.16 and 2.17]{MW}}] 
\label{lem:VK}
Any homotopically trivial cycle embedded in~$X^1$ is the boundary cycle of a disc diagram in $X$. Any minimal disc diagram in~$X$ is reduced.
\end{lemma}

\begin{lemma}[{\cite[Thm~B]{HL}}] 
\label{lem:main interior valence}
If $X$ is $7$-located and locally $5$-large, then so is $D$ for each minimal disc diagram $D\to X$. In other words, $D$ has no
\begin{itemize}
        \item interior vertices of valence $3$ or $4$, or
       \item neighbouring interior vertices with valences $5$ and $5$ or $6$.
\end{itemize}
\end{lemma}

\begin{cor}
\label{cor:5large}
Each $D$ as in Lemma~\ref{lem:main interior valence} 
\begin{enumerate}[(i)]
\item
is $5$-large, and
\item 
each of its induced $5$-cycles is the boundary of a $5$-wheel.
\end{enumerate}
\end{cor}

\begin{proof} Let $D'\subseteq D$ be a simplicial disc with boundary $\omega$ a $4$-cycle or a $5$-cycle. Following the proof of \cite[Cor~4.7]{HL}, we metrise each $5$-wheel of $D'$ as the regular Euclidean pentagon $P$ of side length $1$, and each remaining triangle of $D'$ as the regular Euclidean triangle of side length $1$. Consequently, all triangles have area $\geq \frac{1}{5}\mathrm{Area}(P)>0.34$ (which is larger than in \cite[Cor~4.7]{HL}, where $4$-wheels were allowed). Since $D'$ is $\mathrm{CAT}(0)$ \cite[Thm~4.6]{HL}, by \cite[III.2.17]{BH} the area of $D'$ is $\leq \frac{|\omega|^2}{4\pi}$. Consequently, the number of triangles of $D'$ is $\leq\Big \lfloor\frac{|\omega|^2}{1.36\pi}\Big \rfloor$, which equals $3$ for $|\omega|=4$, and equals $5$ for $|\omega|=5$. Thus $\omega$ is not induced or it is the boundary of a $5$-wheel, as desired.
\end{proof}

\begin{remark}
The $\kappa'$ method from the proof of Proposition~\ref{prop:lunar2} can be used to give an alternative proof of Corollary~\ref{cor:5large}, since in a minimal counterexample to Corollary~\ref{cor:5large} each valence $5$ interior vertex has non-positive $\kappa'$.
\end{remark}

\section{Lunar diagrams}

We metrise the $1$-skeleton of each simplicial complex with the path metric, where all the edges have length $1$.

\begin{defn}
    Let $D$ be a simplicial disc and let $x,v$ be vertices of $\partial D$.
    We say that $D$ is \emph{lunar (between $x$ and $v$)}, if the two distinct injective paths $\gamma_1$ and $\gamma_2$ in~$\partial D$ from $x$ to $v$ are geodesics in $D^1$.
    
    More generally, let $X$ be a simplicial complex and let $x$ and $v$ be distinct vertices of $X$. Suppose that $\gamma_1$ and $\gamma_2$ are geodesics from $x$ to $v$ in $X^1$ 
    that are disjoint except at the endpoints.
    A minimal disc diagram $D\to X$ with boundary $\gamma_1 \cup \gamma_2$ is a \emph{lunar disc diagram (between $x$ and $v$)}. Note that if $D\to X$ is a lunar disc diagram, then $D$ is a lunar simplicial disc.

Let $D$ be a lunar simplicial disc. Then for a vertex $u$ of~$\partial D$, an interior vertex of~$D$ of valence $5$ is \emph{$u$-exposed},  
if it is a neighbour of both neighbours of $u$ 
in $\partial D$. 
\end{defn}

Let $D$ be a lunar simplicial disc between $x$ and $v$ that is $7$-located and locally $5$-large. By Corollary~\ref{cor:5large}(i), if $z$ is $u$-exposed, and the neighbours of $u$ in $\partial D$ are denoted $u_1,u_2$, then the cycle $u_1uu_2z$ has a diagonal. (Note that if $u\neq x,v$, then $z$ is a neighbour of $u$.) Consequently, for each $u$ there is at most one $u$-exposed vertex.

\begin{prop} \label{prop:lunar2}
Let $D$ be a lunar simplicial disc between $x$ and $v$ and let $v_1,v_2$ be the neighbours of $v$ in $\partial D$. If $D$ is $7$-located and locally $5$-large, then 
\begin{enumerate}[(i)]
\item either $v_1$ and $v_2$ are neighbours and have a common neighbour closer to $x$ than $v_1,v_2$, or
\item $v_1$ and $v_2$ are not neighbours, and there is a $v$-exposed neighbour $v'$ of $v_1,v,v_2,$ whose remaining neighbours $v_1',v_2'$ are closer to $x$ than $v_1,v_2$.
\end{enumerate}
\end{prop}

In the proof, we need the following.  We are using the path metric $d$ on $D^{1}$, the path metric~$d'$ on~$D'^1$, and the path metric~$d''$ on~$D''^1$.

\begin{lemma} \label{lem:nice lunar diagram}
Let $D$ be a lunar simplicial disc between $x$ and $v$ and let $v_1,v_2$ be the neighbours of $v$ in $\partial D$. If $D$ is $7$-located and locally $5$-large, then there is a lunar simplicial disc $D'\subseteq D$ between $z$ and $v$ such that 
    \begin{enumerate}
          \item the path $v_1vv_2$ lies in $\partial D'$,
          \item the function $d(\cdot,x)-d'(\cdot,z)$ is constant on all the vertices of $D'$ at $d'$-distance $\leq 2$ from $v$,
        \item each vertex on $\partial D'$ has valence at least $4$, except possibly for $z$, $v$, 
        $v_1$, or $v_2$, 
               \item each interior vertex of $D'$ of valence $5$ that is a neighbour of $\geq 3$ vertices of~$\partial D'$ is $u$-exposed in $D'$ with $u\in \{z,v,v_1,v_2\}$.
     \end{enumerate}  
\end{lemma}

\begin{remark}
\label{rem:polite}
By (1) and (2), if $D'$ satisfies Proposition~\ref{prop:lunar2}(i) or (ii), then so does~$D$.
\end{remark}

In the proof of Lemma~\ref{lem:nice lunar diagram}, we will use the following.

\begin{lemma}
\label{rem:cut}
Let $D'\subseteq D$ be connected simplicial complexes with vertices $z'\in D',z\in D$. Let $I\subset D'$ be a subcomplex, and let $C$ be a connected component of $D'\setminus I$ that is equal to a connected component of $D\setminus I$, with $z,z'\notin C$. If the function $d(\cdot,z)-d'(\cdot,z')$ is constant on all the vertices of $I$, then it is constant on all the vertices of $I\cup C$. 
\end{lemma}
\begin{proof}
Denote the constant value of $d(\cdot,z)-d'(\cdot,z')$ on the vertices of $I$ by $K$. Let $\overline C$ be the closure of $C$ in $D$, and let $d_C$ denote the path metric on the $1$-skeleton of $\overline C$.
For each vertex $w$ of $C$, each path in $D^1$ from~$w$ to~$z$ decomposes into a (nontrivial) path in $\overline C^1$ from~$w$ to some $i\in I$  and a (possibly trivial) path in $D^1$ from $i$ to $z$. Thus we have $$d(w,z)=\min_{i\in I}\big(d_C(w,i)+d(i,z)\big).$$
Analogously, 
$$d'(w,z')=\min_{i\in I}\big(d_C(w,i)+d'(i,z')\big).$$
Since the last terms differ always by $K$, we have $d(w,z)=d'(w,z')+K$, as desired.
\end{proof}

\begin{lemma}
\label{lem:geodesic}
Let $D$ be a lunar simplicial disc between $z$ and $v$, and let $D'\subseteq D$ be 
a simplicial disc with $\partial D'$ a concatenation of paths $\gamma_1',\gamma_2'$ from $z'$ to $v$. Suppose that for $j=1,2$ there is a vertex $u_j$ on $\gamma'_j$ with $d(u_j,z)-d'(u_j,z')=d(v,z)-d'(v,z')$ and such that $\gamma_j'$ is a concatenation $\alpha_j\beta_j$ of the subpath $\alpha_j=vu_j$ of the geodesic $\gamma_j$, and a geodesic $\beta_j$ of $D'^1$. Then each $\gamma'_j$ is a geodesic in $D'$ and so $D'$ is a lunar simplicial disc between $z'$ and $v$.  
\end{lemma}
\begin{proof}
By the hypothesis, for each $j$ we have $$d'(v,z')=d(v,z)-d(u_j,z)+d'(u_j,z')=|\alpha_j|+|\beta_j|,$$ and so $\alpha_j\beta_j$ is a geodesic, as desired.
\end{proof}

\begin{proof}[Proof of Lemma~\ref{lem:nice lunar diagram}]
Let $D'\subseteq D$ be a lunar simplicial disc of minimal area between $z$ and $v$ satisfying (1) and (2), among all the vertices $z$ of $D$. We denote by $\gamma_1', \gamma_2'$ the geodesics in $\partial D'$ from $z$ to $v$.

To verify (3), suppose without loss of generality that $z_0\in \gamma'_1\setminus \{z,v,v_1\}$ has valence $\leq 3$, which then equals $3$, since $\gamma_1'$ is a geodesic in $D'^1$. Let $z_1,z_2,z_3$ be the consecutive neighbours of $z_0$ in $D'$, where $z_1$ separates $z_0$ from $v$ on $\gamma_1'$.
If $z_2\notin \gamma_2'$, then let $D''$ be the lunar simplicial disc between $z$ and $v$ obtained from $D'$ by removing the \emph{open star} of $z_0$, which is the union of $z_0$ and all the simplices of $D'$ containing $z_0$ in their closure. Then $D''$ still satisfies (1) and (2), which contradicts the minimality hypothesis. If $z_2\in \gamma_2'$, then let $D''$ be the simplicial disc that is the closure of the connected component of $v$ in $D'\setminus z_1z_2$. Observe that $d'(z,z_1)=d'(z,z_2)+1$, and so by Lemma~\ref{lem:geodesic} applied to $D''\subset D'$ with $z'=u_2=z_2,u_1=z_1$, we have that $D''$ is a lunar simplicial disc between $z_2$ and~$v$. Applying Lemma~\ref{rem:cut} to $D''\subset D', I=z_1z_2,$ and $C=D''\setminus I$, we obtain that $d'(\cdot,z)-d''(\cdot,z')$ is constant on the vertices of~$D''$. Thus $d(\cdot,x)-d''(\cdot,z')=\big(d(\cdot,x)-d'(\cdot,z)\big)+\big(d'(\cdot,z)-d''(\cdot,z')\big)$ is constant on the vertices of~$D''$ at $d''$-distance $\leq 2$ from $v$. Hence $D''$ still satisfies (1) and (2), which contradicts the minimality hypothesis.

To verify (4), let $z_0$ be an interior vertex of $D'$ of valence $5$ that is a neighbour of $m\geq 3$ vertices of $\partial D'$. 

If $m=5$, then by (3) $z_0$ is $v$-exposed in $D'$. The same holds for $m=4$, unless $z_0$ has exactly $2$ neighbours 
on each $\gamma_i'$, which are distinct from $z,v$, and consecutive by Corollary~\ref{cor:5large}(i). We will discuss this possibility below, together with the case $m=3$. Namely, if $m=3$, then $z_0$ is $u$-exposed in $D'$ with $u\in \{z,v,v_1,v_2\}$, unless it has, say, two consecutive neighbours $z_1,z_2$ on $\gamma_1'$ and a neighbour $z_3$ on $\gamma_2'$, all of which are distinct from $z,v$. We can assume $d'(z_1,z) = d'(z_2,z)+1$. 

In each of the cases below we will prove that $z_0$ is $u$-exposed in $D'$, with $u\in \{z,v,v_1,v_2\}$, or we will reach a contradiction by finding a properly contained lunar simplicial disc $D''\subsetneq D'$ between a vertex~$z'$ of~$D'$ and~$v$ satisfying (1) and (2) and hence contradicting the minimality hypothesis.
In fact, property (1) will be always following directly from the construction, so we will not be mentioning it. The simplicial disc $D''$ will be always a lunar simplicial disc between $z'$ and $v$ by Lemma~\ref{lem:geodesic} applied to $D''\subsetneq D'$ and $u_j\in I\cap \gamma'_j$, as before. The path $\beta_j=u_jz'\subset \partial D''$ will have always length $0,1,$ or $2$ and it will be clear that it is a geodesic, as desired. Thus we will be only focusing on justifying (2) in each case.
We will be applying Lemma~\ref{rem:cut} with $C$ the connected component of $D''\setminus I$ containing $v$.

Consider the \emph{top} and \emph{bottom} connected components of $D'$ minus the path~$z_1z_0z_3$ containing $v$, and $z$, respectively. Each of the two neighbours of $z_0$ distinct from $z_1,z_2,z_3$ is \emph{top} (resp.\ \emph{bottom}) if it lies in the top (resp.\ bottom) component. 

If $z_0$ has no top neighbours, and $d'(z_1,z)\neq d'(z_3,z)$, then let $D''$ be the lunar simplicial disc between $z'=z_1$ or $z_3$ and $v$ that is the closure of the connected component $C$ of $v$ in $D'\setminus z_1z_3$. Applying Lemma~\ref{rem:cut} to $D''\subsetneq D'$ and $I=z_1z_3$ contradicts the minimality hypothesis. If $z_0$ has no top neighbours, and $d'(z_1,z)= d'(z_3,z)$, then let $D''$ be the lunar simplicial disc between $z'=z_0$ and $v$ that is the closure of the connected component of $v$ in $D'$ minus the path $z_1z_0z_3$. Applying Lemma~\ref{rem:cut} to $D''\subsetneq D'$ and $I=z_1z_3$, we have that $d'(\cdot,z)-d''(\cdot,z')$ is constant on all the vertices of $D''$ except possibly for $z_0$. However, if $z_0$ is not $v$-exposed in $D'$, then $d''(v,z_0)>2$, and so $D''$ still satisfies (2).

If $z_0$ has no bottom neighbours, and $d'(z_2,z)\neq d'(z_3,z)$, then let $D''$ be the lunar simplicial disc between $z'=z_2$ or $z_3$ and $v$ that is the closure of the connected component of~$v$ in $D'\setminus z_2z_3$. Applying Lemma~\ref{rem:cut} to $D''\subsetneq D'$ and $I=z_2z_3$ contradicts the minimality hypothesis. If $z_0$ has no bottom neighbours, and $d'(z_2,z)= d'(z_3,z)$, then let $z'$ be the common neighbour of $z_2,z_3$ distinct from $z_0$. Let $D''$ be the lunar simplicial disc between $z'$ and $v$ that is the closure of the connected component of~$v$ in $D'$ minus the path $z_2z'z_3$. Applying Lemma~\ref{rem:cut} to $D''\subset D'$ and $I=z_2z_3$, we obtain again that $D''$ satisfies (2). Note that $D''\neq D'$, unless $z_0$ is $z$-exposed in $D'$.

Thus from now on we can assume that $z_0$ has exactly one top neighbour $z_t$ and one bottom neighbour $z_b$ (one of which might lie in $\partial D'$ if $m=4$). Let $k=d'(z_0, z)$. By the triangle inequality, we have $d'(z_1,z)=k+1$ or $k$, and $d'(z_3,z)=k+1,k$ or $k-1$. 

\begin{enumerate}
\item[\underline{Case 1}:] $d'(z_1,z) = k+1$. 
\item[(a)] $d'(z_3,z) = k+1$. 
Let $D''$ be the lunar simplicial disc between $z'=z_0$ and~$v$ that is the closure of the connected component of $v$ in $D'$ minus the path $I=z_1z_0z_3$. We can then apply Lemma~\ref{rem:cut} to $D''\subsetneq D'$ and $I$ as above. 
          
\item[(b)] $d'(z_3,z) = k$. Then $d'(z_b,z) = k-1$.
We take $z'=z_b$ and $D''$ to be the closure of the connected component of $v$ in $D'$ minus the path $I=z_2z_bz_3$. By Lemma~\ref{rem:cut} the function $d'(\cdot,z)-d''(\cdot,z')$ is again constant on all the vertices of $D''$. Note that $D''\neq D'$, unless $z_0$ is $z$-exposed in $D'$.
        
\item[(c)] $d'(z_3,z) = k-1$. We take $z' = z_3$ and $D''$ the closure of the connected component of $v$ in $D'$ minus the path $I=z_1z_0z_3$.

\item[\underline{Case 2}:] $d'(z_1,z) = k$. 
\item[(a)] $d'(z_3,z) = k+1$. We take $z' = z_2$ and $D''$ the closure of the connected component of $v$ in $D'$ minus the path $I=z_2z_0z_3$.

\item[(b)] $d'(z_3,z) = k$. 
Let $z'$ be the common neighbour of $z_2$ and $z_b$ distinct from $z_0$. let $D''$ be the lunar simplicial disc between $z'$ and $v$ that is the closure of the connected component of $v$ in $D'$ minus the path $z_2z'z_bz_3$. Let $I\subset D''$ be the path $z_2z_0z_3$. Note that $d'(\cdot,z)-d''(\cdot,z')\equiv k-2$ on $I$. By Lemma~\ref{rem:cut}, the function $d'(\cdot,z)-d''(\cdot,z')$ is constant on all the vertices of $D''$ except possibly at $z'$ and $z_b$. Note that $d''(z',v)>2$. If $d''(z_b,v)=2$, then $d'(z,v)=k+1$ and so $v$ is a common neighbour of $z_1$ and $z_3$. Thus $z_0$ is $v$-exposed in $D'$. Note that $D''\neq D'$, unless $z_0$ is $z$-exposed in $D'$.
 
\item[(c)] $d'(z_3,z) = k-1$. We take $z'=z_b$ and $D''$ to be the closure of the connected component of $v$ in $D'$ minus the path $z_2z_bz_3$. Let $I$ be the path $z_2z_1z_tz_3$. Note that $d'(\cdot,z)-d''(\cdot,z')\equiv k-2$ on $I$. By Lemma~\ref{rem:cut}, the function $d'(\cdot,z)-d''(\cdot,z')$ is constant on all the vertices of $D''$  except possibly at $z'$ and $z_0$. Note that $d''(z',v)>2$. If $d''(z_0,v)=2$, then $d'(z,v)=k+1$ and so $v$ is a neighbour of $z_1$ and connected by the path $vv_2z_3$ of length $2$ with~$z_3$. By Corollary~\ref{cor:5large}(ii) applied to the $5$-cycle $z_0z_1vv_2z_3$, we have that either $z_t=v_2$, and so $z_0$ is $v$-exposed, or this $5$-cycle is the boundary of a $5$-wheel with centre~$z_t$, and so both $z_t,z_0$ have valence $5$, which is a contradiction. Again $D''\neq D'$, unless $z_0$ is $z$-exposed in $D'$.
\end{enumerate}
\end{proof}

\begin{proof}[Proof of Proposition~\ref{prop:lunar2}]
By Remark~\ref{rem:polite}, and Lemma~\ref{lem:nice lunar diagram}, we can assume that $D=D'$ and satisfies Lemma~\ref{lem:nice lunar diagram}(3,4).
For any interior vertex $w$ of $D'$, let $\kappa(w)=6$ minus the valence of $w$. For $w$ in $\partial D'$, let $\kappa(w)=4$ minus the valence of $w$. By the combinatorial Gauss--Bonnet theorem (see e.g.\ \cite[Thm~4.6]{MW}), the sum of all $\kappa(w)$ equals $6$. 
For each interior vertex $w$ of valence $5$, let $\kappa'(w)=\kappa(w)-\frac{N(w)}{3}$, where $N(w)$ is the number of the interior neighbours of $w$ (all of which have valence $\geq 7$ by Lemma~\ref{lem:main interior valence}). For each interior vertex $w$ of valence $\geq 7$, let $\kappa'(w)=\kappa(w)+\frac{N_5(w)}{3}$, where $N_5(w)$ is the number of the interior neighbours of $w$ of valence~$5$. We let $\kappa'(w)=\kappa(w)$ for the remaining $w$. Note that the sum of $N(w)$ over interior $w$ of valence $5$ is equal to the sum of $N_5(w)$ over interior $w$ of valence $\geq 7$, since in each of these sums we count the number of pairs of interior neighbours of valence $5$ and~$\geq 7$. 
Thus the sum of all $\kappa'(w)$ equals $6$ as well.

\smallskip
\noindent \textbf{Claim.} \emph{$\kappa'$ is non-positive except possibly at
\begin{itemize}
\item $x,v$, where it is $\leq 2$,
\item $v_i$, where it is $\leq 1$,
\item $u$-exposed vertices, for $u\in \{x,v,v_1,v_2\}$, where it is $\leq 1$. It is $<1$ if the $u$-exposed vertex has at least one interior neighbour.
\end{itemize}}

\smallskip

Indeed, if an interior vertex $w$ of valence $5$ is not $u$-exposed for $u\in \{x,v,v_1,v_2\}$, then by (4) we have $N\geq 3$, and so $\kappa'(w)=\kappa(w)-\frac{N}{3}\leq 1-1$. On the other hand, if 
an interior vertex $w$ has valence $7$, then by Lemma~\ref{lem:main interior valence} we have $N_5\leq 3$ and so $\kappa'(w)=\kappa(w)+\frac{N_5}{3}\leq -1+1$, and if it has valence $k\geq 8,$ then $N_5\leq \frac{k}{2}$ and so $\kappa'(w)=\kappa(w)+\frac{N_5}{3}\leq 6-k+\frac{k}{6}=6-\frac{5k}{6}<0$. This justifies the Claim.

\smallskip

If there is a $v$-exposed vertex, then we call it $v'$. If such a vertex does not exist, but there are $v_j$-exposed vertices, then we call them $v_j'$. Note that if $v'$ exists, then we do not define $v_j'$, because if an $v_j$-exposed vertex exists, then it equals $v'$.
If there is an $x$-exposed vertex, then we call it $x'$. 
By the Claim, the sum of $\kappa'$ over all the vertices except at $v,v',v_j,v_j'$ is at most $3$. Consequently, the sum of $\kappa'$ at $v,v',v_j,v_j'$ is at least $3$.

To verify (i) or (ii) it suffices to check that 
\begin{enumerate}[(i)]
\item $\kappa'(v)=2$ and $\kappa(v_j)=1$ for some $j$, or
\item $v'$ exists and $\kappa'(v)=\kappa'(v_1)=\kappa'(v_2)=1$.
\end{enumerate}

We will justify that the only remaining possibilities, up to a symmetry, are:

\begin{enumerate}[(a)]
\item $\kappa'(x)=\kappa'(v)=2,\kappa'(x')=\kappa'(v')=1, \kappa'(v_1)=\kappa'(v_2)=0$, \item $\kappa'(x)=2, \kappa'(v)=\kappa'(v_1)=\kappa'(x')=\kappa'(v')=1, \kappa'(v_2)=0$, 
\item $\kappa'(x)=2, \kappa'(v)=\kappa'(v_1)=\kappa'(v_2)=\kappa'(x')=1$, and there are no $v',v_j'$, 
\item $\kappa'(x)=2,\kappa'(v_1)=\kappa'(v_2)=\kappa'(v'_1)=\kappa'(x')=1,\kappa'(v)=0$ and there are no $v',v_2'$,
\item $\kappa'(x)=2, \kappa'(v_1)=\kappa'(v_2)=\kappa'(v'_1)=\kappa'(v'_2)=\kappa'(x')=1, \kappa'(v)=-1$.
\end{enumerate}

Indeed, consider first the case where $v'$ exists. If $\kappa'(v)=2$, then $\kappa'(v_1),\kappa'(v_2)$ are nonpositive. This brings us to case (a). If $\kappa'(v)=1$, then we have~(ii), or, say, we have $\kappa'(v_2)\leq 0$. This brings us to case (b).

Second, suppose that there is no $v'$. Note that if one of the $v'_j$ exists, then $\kappa'(v)\leq 0$. If both $v'_j$ exist, then they are not neighbours by Lemma~\ref{lem:main interior valence}, and so $\kappa'(v)\leq -1$. This brings us to (i) or cases (c--e). 

However, in all cases (a--e), by (3), the vertex $x'$ has at least one interior neighbour, which contradicts $\kappa'(x')=1$.
\end{proof}

\section{Contractibility}

Here, as before, we do not assume our simplicial complexes to be locally finite or finite-dimensional. 

\begin{lemma} \label{lem:small balls are contractible}
    Suppose that $K$ is a nonempty flag simplicial complex 
    \begin{enumerate}
        \item of diameter $\leq 2$,
        \item $5$-large, and 
        \item such that any induced $5$-cycle is the boundary of a $5$-wheel of $K$.
    \end{enumerate}
    Then $K$ is contractible.
\end{lemma}

In the proof, we will use the following.

\begin{lemma}[{\cite[Lem~8.11]{HP}}]
\label{lem:HP}
Let $K$ be as in Lemma~\ref{lem:small balls are contractible}. Then for any pair of simplices of $K$ with vertex sets $A_1,A_2$, there is a vertex $a$ of $K$ that is a neighbour of or equal to each of the elements of $A_1\cup A_2$. 
\end{lemma}

Consequently, both  $A_1\cup \{a\}$ and $A_2\cup \{a\}$ span a simplex. 

\begin{proof}[Proof of Lemma~\ref{lem:small balls are contractible}]
By the Whitehead theorem, it suffices to show that any finite subcomplex $K'$ of~$K$ is contained in a contractible subcomplex $K''$ of $K$. We consider all the subsets $V_0,\ldots, V_n$ of the vertex set of $K'$ that span a simplex of~$K$. Let~$M_0$ be the simplex spanned on $V_0$. Using Lemma~\ref{lem:HP}, we construct inductively simplices~$M_1,\ldots, M_n$ so that $M_i\supseteq M_{i-1}$ and $M_i$ contains a vertex $a_i$ such that $V_i\cup \{a_i\}$ spans a simplex. Let $K''$ be the span of the union of $K'$ and $M_n$. Note that $K''$ is flag and each maximal simplex of $K''$ intersects $M_n$. Then $K''$ is contractible (which follows essentially from gradually collapsing $K''$ to $M_n$, see e.g.\ \cite[Lem~8.13]{HP}). 
\end{proof}

A subcomplex $C$ of a simplicial complex $K$ is \emph{$3$-convex} if it is a connected induced subcomplex of $K$ and for every path~$abc$ with vertices $a,c$ in $C$ at distance $2$ in the $1$-skeleton $K^1$, we have that $b$ also belongs to $C$.

\begin{remark} \label{lem:3-convex is contractable}
    Let $K$ be as in \Cref{lem:small balls are contractible}. If $C$ is a $3$-convex subcomplex of $K$, then~$K$ also satisfies the hypotheses of \Cref{lem:small balls are contractible}.
\end{remark}

\subsection{Downward links}
In the entire subsection, we assume that $X$ is a simply connected, $7$-located (in particular, flag) and locally $5$-large simplicial complex.

We fix a basepoint vertex $x$ of $X$. The \emph{ball} $B_n(x)$ (resp.\ the \emph{sphere} $S_n(x)$) is the subcomplex of $X$ spanned by all the vertices at distance $\leq n$ (resp.\ $=n$) from $x$ in~$X^1$.
Let $n > 0$ and let $\sigma$ be a simplex contained in $S_n(x)$. The \emph{link} of $\sigma$ is the intersection of the links of all the vertices of $\sigma$, treated as subcomplexes of $X$. The intersection $K(\sigma)$ of the link of $\sigma$ with $S_{n-1}(x)$ is the \emph{downward link} of $\sigma$. We will prove in Proposition~\ref{prop:nonempty} that $K(\sigma)$ is nonempty, which is obvious only in the case where $\sigma$ is a vertex.

Our goal is to show that downward links satisfy the hypotheses of \Cref{lem:small balls are contractible}, and so they are contractible. Note that they are flag, since they are induced subcomplexes of the flag simplicial complex $X$. To start with, let $\sigma=v$ be a vertex. Note that $K(v)$
satisfies \Cref{lem:small balls are contractible}(2) since $X$ is locally $5$-large.

From Proposition~\ref{prop:lunar2} it follows that $K(v)$
satisfies \Cref{lem:small balls are contractible}(1), and more generally:

\begin{cor} \label{prop: diam of K}
Let $v_1,v_2$ be vertices of $K(v)$.
\begin{enumerate}[(i)]
    \item If $v_1$ and $v_2$ are neighbours, then $K(v_1)$ intersects $ K(v_2)$.
    \item If $v_1$ and $v_2$ are not neighbours, then 
    they have a common neighbour $v'$ in~$K(v)$ such that there is an edge $v_1'v_2'$ with $v_1'$ in $K(v_1v')$ and $v_2'$ in $K(v_2v')$.
\end{enumerate}
\end{cor}

\begin{prop}
    Any induced $5$-cycle in $K(v)$ is the boundary of a $5$-wheel in~$K(v)$.
\end{prop}

\begin{proof}
    Let $\gamma=v_2w_2w_1v_1u$ be an induced $5$-cycle in $K(v)$. 
    Let $v', v_1',v_2'$ be as in \Cref{prop: diam of K}(ii). Since $K(v)$ is $5$-large, $u$ and $v'$ are either equal or neighbours. If $v'$ is a neighbour of $w_1$ or $w_2$, then, since $K(v)$ is $5$-large, we have that $(v',\gamma)$ is the required $5$-wheel. Otherwise, $W_1 = (v,v_2w_2w_1v_1v')$ is a $5$-wheel. Since $W_2=(v',v_2v_2'v_1'v_1v)$ is also a $5$-wheel, $(W_1,W_2)$ 
    is a $(5,5)$-dwheel, and so all the vertices of $W_1\cup W_2$ have a common neighbour $y$ of $X$. Since $y$ is a neighbour of both $v$ and~$v'_1$, we have that $y$ is a vertex of $K(v)$. Since $K(v)$ is $5$-large, considering the cycle $v_1uv_2y$ we obtain that $y$ is also a neighbour of $u$. Thus $(y,\gamma)$ is the required $5$-wheel.
  
\end{proof}

\begin{cor} \label{prop:vertex link contractible}
    Each $K(v)$ satisfies the hypotheses of \Cref{lem:small balls are contractible}, and so it is contractible.
\end{cor}

\begin{prop}
\label{prop:nonempty}
   Let $n > 0$, and let $\sigma$ be a simplex of $S_n(x)$. Then $K(\sigma)$ is nonempty.
\end{prop}

\begin{proof}
    Suppose first that $\sigma$ is an edge of $S_n(x)$ with vertices $v_1$ and $v_2$. We may obtain a new complex $X'$ by artificially adding to $X$ a vertex $v$ and a triangle $vv_1v_2$.
    This does not affect local $5$-largeness or $7$-location, so the proposition follows from \Cref{prop: diam of K}(i) applied to $v$ in $X'$.

    Now suppose $\dim(\sigma) > 1$. We fix two distinct vertices $v$ and $y$ of $\sigma$. Let $\sigma'$ be the subsimplex of $\sigma$ spanned on all the vertices except for $y$, and let $e=vy$.
    By induction, 
    we have vertices $v_1$ in $K(\sigma')$ and $v_2$ in $K(e)$.
    If neither $v_1$ nor $v_2$ lie in~$K(\sigma)$, then there is $u\neq v$ in $\sigma'$ that is not a neighbour of $v_2$, and $y$ is not a neighbour of $v_1$. By the $5$-largeness of the link of $v$, the vertex $v_1$ is not a neighbour of $v_2$. Let $v',v'_1,v'_2$ be the vertices from Corollary~\ref{prop: diam of K}(ii). Note that if $v'$ is a neighbour of $y$, then by the $5$-largeness of the link of $v$ it lies in $K(\sigma')$. Thus we can assume that $v'$ is not a neighbour of $y$ and so $W_1=(v, v_2yuv_1v')$ is a $5$-wheel. Since $W_2=(v',v_2v_2'v_1'v_1v)$ is also a $5$-wheel, $(W_1,W_2)$ 
    is a $(5,5)$-dwheel, and so all the vertices of $W_1\cup W_2$ have a common neighbour $z$, which lies in $K(v)$. By the $5$-largeness of the link of $v$, the vertex $z$ lies in $K(\sigma')$, and hence also in $K(\sigma)$.
\end{proof}

\begin{lemma} \label{prop:downward link of simplex is contractible}
   Let $n > 0$, and let $\sigma$ be a simplex of $S_n(x)$. Then for any vertex $v$ of~$\sigma$, the complex $K(\sigma)$ is a $3$-convex subcomplex of $K(v)$.
\end{lemma}

Before the proof, let us note that from Corollary~\ref{prop:vertex link contractible}, Proposition~\ref{prop:nonempty}, Lemma~\ref{prop:downward link of simplex is contractible}, Remark~\ref{lem:3-convex is contractable}, and \Cref{lem:small balls are contractible}, we deduce:

\begin{cor} 
\label{cor:main}
Each $K(\sigma)$ is contractible.  
\end{cor}

\begin{proof}[Proof of Lemma~\ref{prop:downward link of simplex is contractible}.]
Let $abc$ be a path in $K(v)$ with $a,c$ in $K(\sigma)$ at distance $2$ in~$K(v)$. Let $y$ be any vertex of~$\sigma$ distinct from $v$. Applying the $5$-largeness of the link of $v$ to the cycle $abcy$, we obtain that $b$ is a neighbour of $y$. Since this holds for each~$y$, we have that $b$ belongs to $K(\sigma)$, as desired. The connectivity of $K(\sigma)$ follows now from the fact that $K(v)$ has diameter $\leq 2$.
\end{proof}

Below, the \emph{open star} of a simplex $\sigma$ of a simplicial complex $X$ is the union of the interior of the simplex $\sigma$ and all the open simplices of $X$ containing the interior of~$\sigma$ in their closure. 

\begin{proof}[Proof of the Main Theorem]
Let $X$ be a $7$-located locally $5$-large simplicial complex (possibly not locally finite and infinite-dimensional). By passing to the universal cover of $X$, we can assume that $X$ is simply connected.
It suffices to prove that each~$B_n(x)$ is contractible, which we do by induction on $n$. 
Note that each finite subcomplex of $B_n(x)$ is contained in the span of the union of a finite subcomplex of $S_n(x)$ and the entire $B_{n-1}(x)$.
Hence, by the Whitehead theorem, to prove the contractibility of $B_n(x)$ it suffices to show that for each finite induced subcomplex~$A$ of~$S_n(x)$, the span $A_0$ of $A\cup B_{n-1}(x)$ is homotopy equivalent with $B_{n-1}(x)$. To this end, we order the simplices of $A$ in the order of nonincreasing dimension $\sigma_1,\sigma_2,\ldots,\sigma_k$. 
Let~$A_i$ be the (not necessarily induced) subcomplex obtained from~$A_{i-1}$ by removing the open star of $\sigma_i$. Equivalently, $A_{i-1}$ is the union of $A_i$ and the join of $K(\sigma_i)$ with~$\sigma_i$ along the join of $K(\sigma_i)$ with $\partial \sigma_i$. Since both of these joins are contractible, $A_{i-1}$ and $A_i$ are homotopy equivalent \cite[Lem~10.3(i)]{Bj}.
Consequently, $A_0$ is homotopy equivalent with $A_k=B_{n-1}(x)$.
\end{proof}


\begin{bibdiv}
\begin{biblist}

\bib{BC}{article}{
   author={Bandelt, Hans-J\"urgen},
   author={Chepoi, Victor},
   title={Metric graph theory and geometry: a survey},
   conference={
      title={Surveys on discrete and computational geometry},
   },
   book={
      series={Contemp. Math.},
      volume={453},
      publisher={Amer. Math. Soc., Providence, RI},
   },
   isbn={978-0-8218-4239-3},
   date={2008},
   pages={49--86}}

\bib{Bj}{article}{
   author={Bj\"orner, Anders},
   title={Topological methods},
   conference={
      title={Handbook of combinatorics, Vol.\ 1,\ 2},
   },
   book={
      publisher={Elsevier Sci. B. V., Amsterdam},
   },
   isbn={0-444-88002-X},
   date={1995},
   pages={1819--1872}}

\bib{BH}{book}{
   author={Bridson, Martin R.},
   author={Haefliger, Andr\'e},
   title={Metric spaces of non-positive curvature},
   series={Grundlehren der mathematischen Wissenschaften [Fundamental
   Principles of Mathematical Sciences]},
   volume={319},
   publisher={Springer-Verlag, Berlin},
   date={1999}}

\bib{CO}{article}{
   author={Chepoi, Victor},
   author={Osajda, Damian},
   title={Dismantlability of weakly systolic complexes and applications},
   journal={Trans. Amer. Math. Soc.},
   volume={367},
   date={2015},
   number={2},
   pages={1247--1272}}

\bib{H}{article}{
   author={Haglund, Fr\'ed\'eric},
   title={Complexes simpliciaux hyperboliques de grande dimension},
   status={Pr\'epublication Orsay 71},
   date={2003}}


\bib{HL}{article}{
   author={Hoda, Nima},
   author={Laz\u ar, Ioana-Claudia},
   title={$7$-Location, weak systolicity, and isoperimetry},
   journal={Trans. London Math. Soc.},
   volume={12},
   date={2025},
   number={1},
   pages={Paper No. e70016}}

\bib{HP}{article}{
   author={Huang, Jingyin},
   author={Przytycki, Piotr},
   title={$353$-combinatorial curvature and the $3$-dimensional $K(\pi,1)$ conjecture},
   status={available at arXiv:2509.06914},
   date={2025}}

\bib{JS}{article}{
   author={Januszkiewicz, Tadeusz},
   author={\'Swiątkowski, Jacek},
   title={Simplicial nonpositive curvature},
   journal={Publ. Math. Inst. Hautes \'Etudes Sci.},
   number={104},
   date={2006},
   pages={1--85}}

\bib{L}{article}{
   author={Laz\u ar, Ioana-Claudia},
   title={Minimal disc diagrams of $5/9$-simplicial complexes},
   journal={Michigan Math. J.},
   volume={69},
   date={2020},
   number={4},
   pages={793--829}}

\bib{MW}{article}{
   author={McCammond, Jon},
   author={Wise, Daniel T.},
   title={Fans and ladders in small cancellation theory},
   journal={Proc. London Math. Soc. (3)},
   volume={84},
   date={2002},
   number={3},
   pages={599--644}}

\bib{O1}{article}{
   author={Osajda, Damian},
   title={A combinatorial non-positive curvature I: weak systolicity},
   status={available at arXiv:1305.4661},
   date={2013}}

\bib{O2}{article}{
   author={Osajda, Damian},
   title={Combinatorial negative curvature and triangulations of
   three-manifolds},
   journal={Indiana Univ. Math. J.},
   volume={64},
   date={2015},
   number={3},
   pages={943--956}}

\end{biblist}
\end{bibdiv}
\end{document}